\documentclass{amsart}

\usepackage{verbatim}
\usepackage{graphicx}
\usepackage{amsfonts,amsmath,latexsym,amssymb,amsthm}
\usepackage{xcolor}
\usepackage{geometry}
\usepackage{tikz,tikz-cd}
 \usepackage{mathpazo}
\usepackage[mathpazo]{flexisym}
\usepackage{breqn}
\usepackage{hyperref}

\newcounter{theoremcounter}
\newcounter{lemmacounter}

\newcounter{dummycounter}
\newcounter{propcounter}
\newcounter{corcounter}

\newcounter{quescounter}

\newcounter{emptycounter}
\newcounter{defcounter}

\newtheorem{theorem}[theoremcounter]{Theorem}

\newtheorem{question}[quescounter]{Question}

\newtheorem{lemma}[lemmacounter]{Lemma}

\newtheorem{proposition}[propcounter]{Proposition}
\newtheorem{corollary}[corcounter]{Corollary}
\newtheorem{remark}{Remark}
\newtheorem{definition}[defcounter]{Definition}

\newcounter{eqncounter}

\numberwithin{equation}{eqncounter}

\def\IR{\mathbb R}

\def\IZ{\mathbb Z}
\def\IN{\mathbb N}

\def\IQ{\mathbb Q}

\def\siz{\mathfrak{s}}

\def\disc{\mathop{\rm disc}\nolimits}
\def\Frac{\mathop{\rm frac}\nolimits}

\def\Oseen{{\mathcal{O}}}

\def\Qbar{\overline{\IQ}}

\def\Gal{\mathop{{\rm Gal}}\nolimits}
\def\ord{\mathop{{\rm ord}}\nolimits}

\title{Height lower bounds for elements of highly composite rings}

\author[S.~H.~Man]{Siu Hang Man}
\author[N.~Technau]{Niclas Technau}
\author[M.~Widmer]{Martin Widmer}
\author[P.~Yatsyna]{Pavlo Yatsyna}

\address[S.H.M., P.Y.]{Charles University, Faculty of Mathematics and Physics, Department of Algebra, Sokolov\-sk\' a 83, 186~75 Praha~8, Czech Republic}
\address[N.T.]{Department of Mathematics, University of Wisconsin--Madison, 480 Lincoln Drive,
Madison, WI 53706, USA}

\address[M.W.]{Graz University of Technology,
Institute of Analysis and Number Theory,
Steyrergasse 30/II, 8010 Graz, Austria}

\email[S.~H.~Man]{shman@karlin.mff.cuni.cz}
\email[N.~Technau]{technau@wisc.edu}
\email[M.~Widmer]{martin.widmer@tugraz.at}
\email[P.~Yatsyna]{p.yatsyna@mff.cuni.cz}
\date{\today}

\thanks{S.H.M. and P.Y. were supported by the Czech Science Foundation GAČR grant 26-20514S. S.H.M. was supported by the Charles University programme PRIMUS/25/SCI/008.
N.T. was supported by 
the University of Wisconsin–Madison, Office of 
the Vice Chancellor for Research with funding from the Wisconsin Alumni Research Foundation. During the early stages 
of this project, N.T. was supported 
by a Schr\"{o}dinger Fellowship
of the Austrian Science Fund (FWF): project J 4464-N.
P.Y. was supported by the Charles University programme PRIMUS/24/SCI/010.}

\date{\today}
\subjclass{11G50, 11R04}
\keywords{Algebraic number theory, Weil height, Northcott property, Lehmer's conjecture}

\begin{document}

\maketitle
\begin{abstract}
Let $\IQ^{(d)}$ be the composite field of all number fields of degree at most $d$. In 2001 Bombieri and Zannier proved that 
$\IQ^{(2)}$ has the Northcott property and asked what happens for $d\geq 3$. Here we study the absolute Weil height for elements in the composite ring of the rings of integers of such number fields.
In particular, we consider $\IQ^{(3)}$
as the composite field of $\IQ^{(2)}$ and a minimal infinite family of cubic fields, and we show that the composite ring of the rings of integers of these fields does have the Northcott property. Our results follow from  new 
height lower bounds, expressed in terms of the degree. 
Moreover, we introduce a notion of size for subfields of    
$\IQ^{(3)}$. For instance, $\IQ^{(3)}$ has size $1$ and the maximal abelian subfield of $\IQ^{(3)}$ has size $1/2$. We show that there is a subfield of $\IQ^{(3)}$ of size $1$ 
which has the Northcott property.
\end{abstract}

\section{Introduction}
Let $K$ be a number field. For each place $v$ of $K$ we choose the unique representative $|\cdot |_v$ that either
extends the usual archimedean absolute value or one of the usual $p$-adic absolute values on $\IQ$.
We write $K_v$ for the completion of $K$ at $v$ and we write  $d_v=[K_v:\IQ_v]$ for the local degree at $v$.
For $\alpha \in K$ we define the absolute multiplicative Weil height by
\begin{alignat}1\label{def: height}
H(\alpha)=\prod_{v}\max\left\{1,|\alpha|_v\right\}^{\frac{d_v}{[K:\IQ]}},
\end{alignat}
where the product runs over all places $v$ of $K$. Taking the $[K:\IQ]$-th root in the above definition makes the height independent  of the particular choice of the field $K$ containing $\alpha$ (see \cite[Lemma 1.5.2]{BG}). Hence 
the height defines a genuine function on the algebraic numbers $\Qbar$. We write $h(\alpha)=\log H(\alpha)$
for the logarithmic Weil height of $\alpha$. For more details
on the Weil height we refer the reader to \cite[Section 1.5]{BG}.

We are interested in bounding the height of elements of a specific subset $R\subset \Qbar$ from below in terms of their degrees. The question is of particular interest when $R$ is a subring or subfield of $\Qbar$. By Kronecker's theorem $h(\alpha)=0$  if and only if $\alpha$ is a root of unity or zero. We are excluding these cases throughout our discussion.
Lehmer's conjecture from 1933 asserts that 
\begin{alignat}1\label{ineq:Lehmer}
h(\alpha)\geq  \frac{c_0}{[\IQ(\alpha):\IQ]}
\end{alignat}  
for some absolute constant $c_0>0$. If $\IQ(\alpha)/\IQ$ is Galois,
then (\ref{ineq:Lehmer}) holds by a result
of Amoroso and David \cite{AmDa}. The latter was improved by Amoroso and Masser  \cite{AmMa} to
\begin{alignat}1\label{ineq:AmMa}
h(\alpha)\geq  \frac{c_\varepsilon}{[\IQ(\alpha):\IQ]^\varepsilon}
\end{alignat}  
for any $\varepsilon>0$.
Assuming $\IQ(\alpha)/\IQ$ is abelian, one even has an absolute lower bound   
\begin{alignat}1\label{ineq:AmDv}
h(\alpha)\geq  \frac{\log 5}{12}
\end{alignat}  
by an older result of Amoroso and Dvornicich \cite{AmDv}. 
Let $\IQ^{(d)}$ be the composite field of all number fields of degree at most $d$.
Bombieri and Zannier \cite{BoZa} 
showed that 
there exists 
a nondecreasing \emph{unbounded} function $F: \mathbb{Z}_{\geq 1}\rightarrow \IR_{\geq 0}$, (depending on $d$)
such that
\begin{alignat}1\label{ineq:BoZa}
h(\alpha)\geq  F([\IQ(\alpha):\IQ])
\end{alignat}
whenever $\IQ(\alpha)/\IQ$ is abelian and contained in $\IQ^{(d)}$.\\

We say a set $S\subset \Qbar$ has property (\ref{ineq:BoZa}) if there exists 
a nondecreasing \emph{unbounded} function $F: \mathbb{Z}_{\geq 1}\rightarrow \IR_{\geq 0}$, 
such that (\ref{ineq:BoZa}) holds for all $\alpha\in S^*$, where $S^*$ is the subset of $S$
deprived by $0$ and all roots of unity.
Property (\ref{ineq:BoZa}) is essentially a reformulation of the following important finiteness property, formally introduced 
by Bombieri and Zannier \cite{BoZa} in 2001.
\begin{definition}[Bombieri, Zannier \cite{BoZa}]
A set $S\subset \Qbar$ 
has the \emph{Northcott property} if  
$
\#\{\alpha\in S\;:\; h(\alpha)\leq X\}<\infty
$
for every $X\geq 0$. For brevity, 
we say that $S$ has (N) to express that $S$ has the Northcott property.
\end{definition}
The Northcott property has many applications (diophantine and beyond) and has been studied intensively over the last 25 years, see, e.g., \cite{WidmerEssNT} for a survey.
But its origin goes back to 1950 when Northcott (see \cite{Northcott50} and also \cite[Theorem 1.6.8]{BG}) showed that sets of uniformly bounded degree have (N). \\

Thanks to Northcott's Theorem the precise relation between property (\ref{ineq:BoZa}) and (N) is now clear. A set $S$ has (N) if and only if it has only finitely many roots of unity and has property (\ref{ineq:BoZa}). 
A set which contains infinitely many roots of unity and is closed under addition contains an infinite sequence
(without roots of unity) with uniformly bounded height\footnote{This follows from the general inequality $h(\alpha+\beta)\leq h(\alpha)+h(\beta)+\log 2$.}, hence cannot have property (\ref{ineq:BoZa}).
Consequently, for sets that are closed under addition  properties (N) and (\ref{ineq:BoZa})
are equivalent. \\

Since 
the maximal abelian subfield $\IQ^{(d)}_{ab}$ 
of $\IQ^{(d)}$ coincides with $\IQ^{(d)}$ when $d=2$ we know that $\IQ^{(2)}$ has (N).
Bombieri and Zannier asked the following question.
\begin{question}[Bombieri, Zannier, \cite{BoZa}]\label{question:BoZa}
Does a bound as in (\ref{ineq:BoZa}) hold for $\IQ^{(d)}$ when $d\geq 3$?
\end{question}
In other words, does $\IQ^{(d)}$ have (N)?
The case $d=3$ is already widely open and presents a major challenge. In this article we will focus on $d=3$, albeit some of our results are 
more general.

We think of $\IQ^{(3)}$ as the compositum of $L_0=\IQ^{(2)}$ and an infinite sequence of cubic fields $K_1,K_2,K_3,\ldots$.
We can assume that the compositum $L_n=L_{n-1}K_n$ is a proper extension of $L_{n-1}$ for every $n$. We are unable to prove that the elements of 
$$\IQ^{(3)}=L_0\prod_n K_n$$ 
satisfy the required height lower bound but we show that such a height lower bound exists for the composite ring of integers of these fields, i.e., for 
\begin{alignat}1\label{ex:ring}
\Oseen_{L_0}\prod_{n=1}^\infty \Oseen_{K_n}.
\end{alignat}
The above claim follows from Corollary \ref{cor: generalheightdegreebound} (see paragraph after Corollary  \ref{cor: generalheightdegreebound}).
  
The following notation  and definitions are assumed throughout this article.
\begin{definition}\label{def notation}
\noindent \begin{itemize}
\item[(a)] For a field $K\subseteq\overline{\mathbb{Q}}$, let $\mathcal{O}_{K}$
be its ring of integers. For a family $S_i$ ($i\in \mathcal{I}\subset \IN$) of subrings of $\Qbar$, we write 
$$
\prod_{i\in \mathcal{I}}S_i
$$ for the composite ring, i.e., the smallest subring of $\Qbar$ containing  $S_i$  for each $i\in \mathcal{I}$. 
\item[(b)] $L_0$ denotes a subfield of $\Qbar$, and $(K_i)_{i}$ is a sequence of number fields of respective degrees $d_i=[K_i:\IQ]>1$.
We write 
\begin{alignat}1\label{def: RnLn}
L_{n}=L_0\prod_{i=1}^nK_i,\quad
R_n=\Oseen_{L_0}\prod_{i=1}^n \Oseen_{K_i}
\end{alignat}
for the composite rings. Note that the former is a field.\footnote{This follows from the fact that if $K\subset \Qbar$ is a field and $\beta_1,\ldots,\beta_n\in \Qbar$ then
$K[\beta_1,\ldots,\beta_n]=K(\beta_1,\ldots,\beta_n)$.}
We denote the composite ring of all $R_n$ by
$$R=\bigcup_n R_n.$$
\item[(c)] We say the family
$L_0, K_1,K_2,K_3,\ldots$ is 
linearly disjoint over $\IQ$ if every finite subfamily is linearly disjoint 
over $\IQ$.
\end{itemize}
\end{definition}
\begin{remark}\label{rem:ldjn}
The linear disjointness of
$L_0, K_1,K_2,K_3,\ldots$ over $\IQ$
is equivalent to 
$[L_n:L_{n-1}]=[K_n:\IQ]$ 
for every $n\in \IN$ (see \cite{12} discussion before Lemma 2.5.6).
\end{remark}
It is slightly more convenient to prove and express the results in terms of the multiplicative height $H(\cdot)$ defined in (\ref{def: height}).
\begin{theorem}\label{thm: heightdegreebound}
Suppose the family
$L_0, K_1,K_2,K_3,\ldots$ is linearly disjoint over $\IQ$, and all degrees  $d_i$ are prime and bounded from above by some $d\in \IN$.
Then there exists $c_d>0$, depending only on $d$, such that for every $\alpha\in R$
\begin{alignat}1\label{ineq: heightdegreebound}
H(\alpha)\geq
c_d \left(\log[L_0(\alpha):L_0]\right)^{\frac{1}{d^3}}.
\end{alignat}
\end{theorem}

Apart from the choice of $L_0$ and $d$ the height lower bound (\ref{ineq: heightdegreebound}) is completely independent of the particular choice of the ring $R$.
For each $L_0$ and $d>1$ let us define $\mathcal{R}_{L_0}(d)$ to be the union of all rings $R$ as in
Definition \ref{def notation} (b) that are satisfying the hypothesis of Theorem \ref{thm: heightdegreebound}. Hence, the height bound (\ref{ineq: heightdegreebound}) remains valid for every element $\alpha\in \mathcal{R}_{L_0}(d)$.
Unfortunately, $\mathcal{R}_{L_0}(d)$ need not be a ring in general. 

If the Galois closures of the fields  $L_0, K_1,K_2,K_3,\ldots$ are also linearly disjoint over $\IQ$ then the condition that the degrees $d_i$ be prime in Theorem \ref{thm: heightdegreebound} can be replaced by the weaker constraint $L_i/L_{i-1}$ has no proper intermediate field. We will explain this more thoroughly in Section \ref{sec: thmghdb}.

In the spirit of Lehmer one can ask what is the correct dependence on the degree $[L_0(\alpha_n):L_0]$ of the lower bound (\ref{ineq: heightdegreebound})?
Here is a simple upper bound
for $L_0=\IQ$ and $K_i=\IQ(p_i^{1/d})$, where $d$ is a prime and $p_i$ is the $i$-th prime. Taking $\alpha_n=\sum_{i=1}^n p_i^{1/d}\in R$, and estimating its height by the maximal modulus of the conjugates gives
$H(\alpha_n)\leq n(n\log n)^{1/d}\leq n^{1.01+1/d}$ (assuming $n$ is big enough). As the degree of $\alpha_n$ is given by $d^n$ we 
conclude 
\begin{alignat}1\label{ineq: uppbound}
H(\alpha_n)\leq (\log [L_0(\alpha_n):L_0])^{1.01+1/d}.
\end{alignat} 
So here the correct exponent lies between
$1/d^3$ and $1.01+1/d$.
Using that for most conjugates of $\alpha_n$ there is much cancellation one can easily reduce the exponent ${1.01+1/d}$ in (\ref{ineq: uppbound}). 
And it is also clear from the proof of Theorem \ref{thm: heightdegreebound} that the exponent $1/d^3$ is far from sharp and can easily be improved.
However, finding the sharp exponent of the lower bound in (\ref{ineq: heightdegreebound}) seems a rather challenging problem, even when $L_0=\IQ$ and $d=2$. \\

Theorem \ref{thm: heightdegreebound} is an immediate consequence of the following result.
Recall that we are assuming $d_i>1$ for all $i\in \IN$ throughout.
\begin{theorem}\label{thm: heightdiscbound}
Suppose the family
$L_0, K_1,K_2,K_3,\ldots$ is linearly disjoint over $\IQ$ and that there exists  $d>1$
so that $d_i=[K_i: \IQ] \leq d$ for all $i\in \IN$.  
If $\alpha\in R$ and $L_{m-1}(\alpha)=L_{m}$, then
\begin{alignat}1\label{ineq:thmmain}
H(\alpha)\geq
c_d |\Delta_{K_m}|^{\frac{1}{2d_m(d_m-1)}}
\end{alignat}
for some $c_d>0$ depending only on $d$.
\end{theorem}
Before we proceed, let us deduce Theorem \ref{thm: heightdegreebound} from Theorem \ref{thm: heightdiscbound}.
\begin{proof}[Proof of Theorem \ref{thm: heightdegreebound}
assuming Theorem \ref{thm: heightdiscbound}]
First we reorder the fields $K_i$ such that the discriminants $|\Delta_{K_i}|$ are increasing. Clearly, this does not affect
the linear disjointness of the family
$L_0, K_1,K_2,K_3,\ldots$ but we get a new sequence of fields $(L_i)_i$. Suppose $\alpha\in L_m\setminus L_{m-1}$. The linear disjointness gives $[L_m:L_{m-1}]=[K_m:\IQ]=d_m$. Since $d_m$ is 
assumed to be prime, we conclude that $L_{m-1}(\alpha)=L_m$. Further, $[L_0(\alpha):L_0]\leq [L_m:L_{0}]\leq d^m$. Using the very crude bound that among $m$
number fields of degree $\leq d$ there must be at least one with discriminant $\gg_d m^{2/d}$ we conclude 
\begin{alignat*}1
H(\alpha)\gg_d m^{\frac{2}{2d_m(d_m-1)d}} \gg_d \left(\log[L_0(\alpha):L_0]\right)^{\frac{1}{d^3}},
\end{alignat*}
as required.
\end{proof}

To compare (\ref{ineq:thmmain}) with bounds from the literature let us assume the ground field $L_0$ is a number field.
Mahler's classical bound applies to any algebraic integer $\alpha$. Let $D$ be its degree and $f_{\alpha/\IQ}$ its minimal polynomial over $\IQ$, then    
\begin{alignat*}1
H(\alpha)\geq
D^{-1/D} |\disc(f_{\alpha/\IQ})|^{\frac{1}{2D(D-1)}}.
\end{alignat*}
Dixit and Kala \cite[Section 3.1]{DixitKala26} have generalized Mahler's inequality from $\IQ$ to arbitrary ground fields.
To relate their bound with (\ref{ineq:thmmain}) let us briefly sketch the idea behind the proof of Theorem \ref{thm: heightdiscbound}. Let $\alpha\in R$ and $L_{m-1}(\alpha)=L_{m}$. For simplicity let us additionally assume $\alpha\in R_m$.
First we express the discriminant of $f_{\alpha/L_{m-1}}$ in terms of the discriminant of $K_m$. 
Choosing an integral generator $\theta_m$
of $K_m$ one can show\footnote{Apply Lemma \ref{lem:heightest} and use that $R_m\subset (1/I_m)R_{m-1}[\theta_m]$.} 
\begin{align}\label{ineq:discf}
\disc(f_{\alpha/L_{m-1}})=\Delta_{\IZ[\theta_m]}\frac{B^2}{I_{m}^{d_m(d_m-1)}}=\Delta_{K_m}\frac{B^2}{I_{m}^{d_m(d_m-1)-2}}
\end{align}
 where $I_m$ is the index $[\Oseen_{K_m}:\IZ[\theta_m]]$ and $B$ is some non-zero algebraic integer over which we have little control.
Now Dixit and Kala's relative Mahler measure inequality \cite[(6)]{DixitKala26} gives 
 \begin{equation}\label{localglobal}
    H(\alpha) \geq   d_{m}^{-\frac{1}{{d_m}}}|N_{L_{m-1} / \IQ} (\disc(f_{\alpha/L_{m-1}}))|^{\frac{1}{2[L_{m-1} : \IQ]d_m(d_m-1)}}.
\end{equation}
For our bound (\ref{ineq:thmmain}) the contribution of $B^2$ to the height has been ignored. Doing the same here  and 
combining (\ref{ineq:discf}) and (\ref{localglobal}) gives
 \begin{equation*}
    H(\alpha) \geq   d_{m}^{-\frac{1}{{d_m}}} \left(\frac{\Delta_{K_m}}{I_{m}^{d_m(d_m-1)-2}}\right)^{\frac{1}{2d_m(d_m-1)}}.
\end{equation*}
One can bound the index $I_m$ from above in terms of the discriminant $\Delta_{K_m}$, e.g., \cite{ThMiHl} yields $I_m\leq 2^{d_m^3}\Delta_{K_m}^{(d_m-2)/4}$
(assuming $K_m$ has no proper subfields). But for $d_m\geq 3$ this results in a trivial lower bound.\\

To prove Theorem \ref{thm: heightdiscbound} a more refined analysis via local heights is needed. For each rational prime $p$ all contributions to the height above $p$ are carefully estimated. Instead of choosing one fixed $\theta_m$ for every $K_m$ we need to pick a suitable generator $\theta_m=\theta_m(p)$ for every prime $p$.
Classical results on common inessential discriminant divisors, e.g., 
a theorem due to von \.{Z}yli\'{n}ski
(see Theorem \ref{thm: Zylinski}), permit us to choose $\theta_m(p)$ with the required properties, at least when $p$ is sufficiently large in terms of $d_m$. Additional difficulties arise from the fact that we cannot assume $R_m=R\cap L_m$  
which means we need to control the degrees $d_i$ for $i\geq m$. 
The latter explains why we want the degrees $d_i$ to be uniformly bounded.

Theorem \ref{thm: heightdiscbound} is also related to Silverman's bound \cite[Theorem 2]{Silverman}
 \begin{equation}\label{ineq: Silverman}
    H(\alpha) \geq   d_m^{-\frac{1}{d_m}}(N_{L_{m-1} / \IQ} (\Delta_{L_m/L_{m-1}}))^{\frac{1}{2[L_{m-1} : \IQ]d_m(d_m-1)}}.
\end{equation}
The bound (\ref{ineq: Silverman}) applies for every $\alpha\in L_m$ with $L_{m-1}(\alpha)=L_m$.
If $\Delta_{K_m}$ and $\Delta_{L_{m-1}}$ are coprime then Silverman's (\ref{ineq: Silverman}) recovers our (\ref{ineq:thmmain}) but if 
$L_m/L_{m-1}$ is unramified then (\ref{ineq: Silverman}) gives nothing at all. 
To establish height lower bounds for elements in an infinite field tower $L=\bigcup_nL_n$ the inequality
(\ref{ineq: Silverman}) is therefore only useful when the field tower is sufficiently ramified at each step. Theorem \ref{thm: heightdiscbound} allows us to get rid of this ramification condition at the expense of replacing the field $L$ by the ring $R$.\\

A slight generalization of Northcott's Theorem shows that sets of uniformly bounded 
degree over any field with Northcott property also have Northcott property (see \cite[Theorem 2.1]{DvZa}).
Using this fact in conjunction with Hermite's Theorem we immediately deduce the following corollary from Theorem \ref{thm: heightdiscbound}.

\begin{corollary}\label{cor: generalheightdegreebound}
Suppose $L_0$ has the Northcott property, the family
$L_0, K_1,K_2,K_3,\ldots$ is linearly disjoint over $\IQ$, and the degrees  $d_i$ are uniformly bounded from above.
Moreover, suppose there is no proper intermediate field between $L_{n-1}$ and $L_n$ for every $n\in \IN$.
Then $R$ also has the Northcott property.
\end{corollary}
If the fields $K_i$ are all cubic and $\IQ^{(2)}\subseteq L_0$ then it suffices\footnote{We can replace $L_0$ by some $L_N$ to assume $L_n\neq L_{n-1}$ for all $n$. 
Since $\IQ^{(2)}\subseteq L_0$ we have $L_n/L_{n-1}$ is Galois for every cubic $K_n$. Hence,
$[L_n:L_{n-1}]=[K_n:\IQ]$ for all $n$. Therefore the family $L_0, K_1,K_2,K_3,\ldots$ is linearly disjoint over $\IQ$, and $d_i=3$ for all $i$.} to assume $L_n\neq L_{n-1}$
for all but finitely many $n\in \IN$ in place of the ``linear disjointness''- condition. This proves the claim just before Definition \ref{def notation}, namely, that (\ref{ex:ring}) 
has the Northcott property.

It is worthwhile to mention that the mere existence of a subring of $\Oseen_{\IQ^{(3)}}$ with (N) whose field of fractions is equal to $\IQ^{(3)}$ is easy to prove. The difficulty in the proof of the above result 
comes from the fact that each of the components in (\ref{ex:ring}) is a \emph{maximal} order.

In fact, every ring contains a subring with (N) and with the same field of fractions.
The proof of this statement is merely an exercise but we have not found the result in the literature and we believe it merits publication.
\begin{proposition}
\label{prop: rings}
Let $S$ be a subring of $\Qbar$. Then there exists a subring $T$
of $S$ such that $T$ has $(N)$ and the field of fractions of $S$ and of $T$ are identical.
\end{proposition}
We are not aware of any examples of fields $K$ such that $\Oseen_K$ has (N) but $K$ does not
(but see \cite[Exemple 4.6]{GaudronRemondSiegel} for a field without (N) whose ring of integers has (N) w.r.t. to a larger ``height'', called house\footnote{The house of an algebraic integer is the maximal modulus of its conjugates.}).\\

Let us very briefly touch on two different applications of Corollary \ref{cor: generalheightdegreebound} and Proposition \ref{prop: rings} for rings of totally real integers.

In 1962 Julia Robinson \cite{Robinson1962} showed that the semi-ring $(\IN_0,0,1,+,\cdot)$ is first-order definable in the 
ring of integers of any totally real field $L$, provided $\Oseen_L$ has (N). Vidaux and Videla \cite{ViVi15AMS} observed that
Robinson's proof also works for arbitrary rings $R$ of totally real integers (not necessarily the ring of integers of a field) with (N).\footnote{In fact, it suffices if $R$ has (N) w.r.t. the house.} Hence, any such ring has undecidable first-order theory. 

Our next application comes from more recent work of the first author, Daans and Kala \cite{DKM24}. 
Let  $R$ be a ring of totally real integers. A totally positive definite quadratic form $Q(x_1,\ldots,x_n)$ with coefficients in $R$ is said to be universal over $R$ if every totally positive element of $R$ is represented by $Q$.
Famously, the sum of four squares is universal over $\IZ$ by Lagrange's Theorem. If $R$   is 
the ring of integers
a number field then
there exists a universal quadratic form over $R$ (see the survey \cite{Kala23} for this and more background). 
On the other hand, if $R$ is \emph{not} contained in a number field then
it is usually difficult to decide whether a universal quadratic form over $R$ exists. The answer is definitely no if $R$ has (N), as was shown in \cite[Theorem 3.2]{DKM24}.   \\

Our next goal is to formulate a quantitative version of Bombieri and Zannier's Question \ref{question:BoZa}. Again we will focus on the case $d=3$.
To this end we introduce the following notion of
``size'' for subfields of $\IQ^{(3)}$. 
Since for $L\subseteq \IQ^{(3)}$ we have ``equality'' if and only if $L$ contains all cubic fields,
we measure the size of a subfield $L$ of $\IQ^{(3)}$ by measuring how many cubic fields $L$ contains.
To this end let $N_L(X)$ be the number of cubic fields in $L$ with modulus of the discriminant no larger than $X$.
\begin{definition}\label{Size}
Let $L$ be a subfield of $\IQ^{(3)}$.
We define the size $\siz(L)$ of $L$  as 
$$\siz(L):=\liminf_{X\rightarrow \infty}\frac{\log^+ N_{L}(X)}{\log N_{\IQ^{(3)}}(X)},$$
where $\log^+(x):=\log(\max\{1,x\})$.
\end{definition}
Note that $\siz(L)\in [0,1]$, $\siz(L_1)\leq \siz(L_2)$ whenever $L_1\subset L_2$, and $\siz(\IQ^{(3)})=1$. 
Here is a quantitative version of Bombieri and Zannier's question.
\begin{question}\label{question: BoZa3'}
How big can a subfield $L$ of $\IQ^{(3)}$ with Northcott property be?
\end{question}
Wright \cite[Theorem I.2]{wright1989distribution}, choosing
$k=\IQ, G=C_3$ there, showed that the number of abelian cubic extensions grows like $X^{1/2}$. 
In contrast, the Davenport-Heilbronn Theorem \cite{81} states that
the total number of cubic extensions grows linearly. Thus the field $\IQ_{ab}^{(3)}$ is a field with Northcott property and has size $1/2$.\\

Using a criterion of the third author \cite[Corollary 4]{WidmerPropN}, one can easily construct larger 
subfields of $\IQ^{(3)}$ with (N), say of size $\geq 3/5$ (see Section \ref{sec: qantBZQ}). But elementary constructions seem insufficient
to provide a complete answer to Question \ref{question: BoZa3'}.
Fortunately, a deep result of Nakagawa provides sufficiently precise information on the distribution of cubic number fields with ``nearly'' prime discriminant. Combining this result with the aforementioned criterion gives the following observation which is proved in Section \ref{sec: qantBZQ}.
\begin{proposition}\label{prop:SL1}
There exists a subfield $L$ of $\IQ^{(3)}$ with the Northcott property and maximal size $\siz(L)=1$.
\end{proposition}

\section{Proof of Proposition \ref{prop: rings}}
Let $L=\Frac(S)$ be the field of fractions of $S$.
If $L$ is a number field then we can take $T=S$.
Hence we can assume $L/\IQ$ is an infinite extension.
We  can find a sequence of algebraic integers $(\theta_n)_n$ such that
$\IZ[\theta_1,\theta_2,\theta_3,\ldots]\subseteq S$ and 
$L=\IQ(\theta_1,\theta_2,\theta_3,\ldots)$.
Let $L_n=\IQ(\theta_1,\ldots,\theta_n)$ 
so that $L=\cup_{n\geq 1} L_n$.
Furthermore, we can assume that 
$L_{n-1}\neq L_n$ for every $n\in \IN$.
Let us define $\alpha_1=\theta_1$ and set $T_1=\IZ[\alpha_1]$. Hence, $\Frac(T_1)=L_1$. 
Now suppose we have constructed $T_{n-1}=\IZ[\alpha_1,\ldots,\alpha_{n-1}]$. Suppose $T_{n-1}$ has field of fractions 
$L_{n-1}$. Then 
$$I_{n-1}=[\Oseen_{L_{n-1}}:T_{n-1}]$$ 
is finite and we set 
$$\alpha_n=nI_{n-1}\theta_n.$$
Hence, the field of fractions of $T_n=T_{n-1}[\alpha_n]$ is equal to $L_n$.
We set  
$$
T=\bigcup_{n\geq 1} T_n=\IZ[\alpha_1,\alpha_2,\alpha_3,\ldots].
$$ 
Since $\theta_n$ is an algebraic integer it follows from Gauss' Lemma that its minimal polynomial 
$$f_{\theta_n/L_{n-1}}(x)=x^{d_n}+a_1x^{d_n-1}+\cdots+a_{d_n}$$ over  $L_{n-1}$ 
lies in $\Oseen_{L_{n-1}}[x]$. And since $L_n\neq L_{n-1}$ it follows that
$d_n>1$. 
Using that $(nI_{n-1})^{d_n}f_{\theta_n/L_{n-1}}(\theta_n)=0$ we find 
\begin{alignat}1\label{eq:Itheta0}
\alpha_n^{d_n}=(nI_{n-1}\theta_n)^{d_n}=-a_1(nI_{n-1})\alpha_n^{d_n-1}-\cdots -(nI_{n-1})^{d_n-1}a_{d_n-1}\alpha_n-(nI_{n-1})^{d_n} a_{d_n}. 
\end{alignat}
Next note that 
\begin{alignat*}1
I_{n-1}\Oseen_{L_{n-1}}\subset T_{n-1}.
\end{alignat*}
Hence, it follows from (\ref{eq:Itheta0}) that
\begin{alignat*}1
\alpha_n^{d_n}\in T_{n-1}+T_{n-1}\alpha_n+\cdots + T_{n-1}\alpha_n^{d_n-1}. 
\end{alignat*}

Suppose $\gamma \in T$. There exists a minimal integer $n$ such that there are polynomials 
$g_0,\ldots,g_{d_n-1}$ in $\IZ[x_1,\ldots,x_{n-1}]$ with
\begin{alignat}1\label{eq:gammarep}
\gamma=\sum_{i=0}^{d_n-1}g_i(\alpha_1,\ldots,\alpha_{n-1})\alpha_n^i\in T_{n-1}+T_{n-1}\alpha_n+\cdots + T_{n-1}\alpha_n^{d_n-1}. 
\end{alignat}
If $g_i(\alpha_1,\ldots,\alpha_{n-1})=0$ for $1\leq i\leq d_n-1$ then $\gamma=g_0(\alpha_1,\ldots,\alpha_{n-1})$.
This leads to an identity as in (\ref{eq:gammarep}) with $n$ replaced by $n-1$, contradicting the minimality of $n$. Hence there exists $1\leq i\leq d_n-1$ with
$g_i(\alpha_1,\ldots,\alpha_{n-1})\neq 0$. Since $L_n=L_{n-1}(\alpha_n)$ has degree $d_n>1$ over $L_{n-1}$
we conclude that $L_{n-1}(\gamma)\neq L_{n-1}$. Hence, there exists a field homomorphism $\sigma:L_{n-1}(\gamma)\to \Qbar$, not the identity, that fixes $L_{n-1}$.
We extend $\sigma$ to $L_n$ and we write $\sigma(\beta)=\beta'$. 
We get
\begin{alignat*}1
0\neq \gamma-\gamma'=\sum_{i=1}^{d_n-1}g_i(\alpha_1,\ldots,\alpha_{n-1})(\alpha_n^i-(\alpha'_n)^i)=
(\alpha_n-\alpha'_n)G(\alpha_1,\ldots,\alpha_n,\alpha_n'),
\end{alignat*}
where  $G(\alpha_1,\ldots,\alpha_n,\alpha_n')$ is some non-zero algebraic integer in the field
$M=L_n(\alpha_n')$. Set $D=[M:\IQ]$. Using standard estimates for the height we get

\begin{alignat*}1
2H(\gamma)^2\geq H(\gamma-\gamma')\geq |N_{M/\IQ}(\gamma-\gamma')|^{\frac{1}{D}}
=|N_{M/\IQ}(\alpha_n-\alpha'_n)|^{\frac{1}{D}}|N_{M/\IQ}G(\alpha_1,\ldots,\alpha_n,\alpha_n')|^{\frac{1}{D}}
\end{alignat*}
Now $\alpha_n-\alpha'_n=nI_{n-1}(\theta_n-\theta'_n)$, and $I_{n-1}(\theta_n-\theta'_n)$ and $G(\alpha_1,\ldots,\alpha_n,\alpha_n')$ are both non-zero elements in $\Oseen_M$. This proves that
$|N_{M/\IQ}(\alpha_n-\alpha'_n)|^{\frac{1}{D}}\geq n$ and $|N_{M/\IQ}G(\alpha_1,\ldots,\alpha_n,\alpha_n')|^{\frac{1}{D}}\geq 1$. 

We have shown that if $\gamma\in T$ has height below $\sqrt{k/2}$ then $\gamma$ lies in $T_k$. Since $T_k$ is contained in the number field $L_k$ we have only finitely many choices for $\gamma$ with height below $\sqrt{k/2}$ by Northcott's Theorem. This proves that $T$ has $(N)$.

\section{Proof of Theorem \ref{thm: heightdiscbound}}

For the proof of Theorem \ref{thm: heightdiscbound} it is convenient to state the following nearly trivial fact as a lemma. 
\begin{lemma}\label{lem:heightest}
Let $K$ be a subfield of $\Qbar$, let $\theta\in \Qbar$, and let $L=K(\theta)$. Suppose that $d=[L:K]>1$, and let $\alpha\in L$ be such that $L=K(\alpha)$.
Let $(\beta_0,\ldots,\beta_{d-1})\in K^d$ be the unique vector with $\alpha =\beta_{d-1}\theta^{d-1}+\cdots +\beta_1\theta +\beta_0$. 
Let $\sigma_1,\ldots,\sigma_d$ be the $d$ distinct embeddings of $L$ into $\Qbar$, fixing $K$. Set
\begin{alignat*}1
s_{k,i,j}=\sum_{r=0}^{k}\sigma_i(\theta)^r\sigma_j(\theta)^{k-r}, \text{ and } 
\gamma=\prod_{1\leq i<j\leq d}\sum_{k=0}^{d-2}\beta_{k+1}s_{k,i,j}.
\end{alignat*}
Then  
\begin{alignat*}1
0\neq \disc_{L/K}(1,\alpha,\ldots,\alpha^{d-1})=\prod_{1\leq i<j\leq d}(\sigma_i(\alpha)-\sigma_j(\alpha))^2 
=\disc_{L/K}(1,\theta,\ldots,\theta^{d-1})\gamma^2.
\end{alignat*}
\end{lemma}
\begin{proof}
As $L=K(\alpha)$ we conclude $0\neq \prod_{1\leq i<j\leq d}(\sigma_i(\alpha)-\sigma_j(\alpha))$. Further, we have 
\begin{alignat*}1
\sigma_i(\alpha)=\beta_{d-1}\sigma_i(\theta)^{d-1}+\cdots +\beta_1\sigma_i(\theta) +\beta_0,
\end{alignat*}
and thus
\begin{alignat*}1
\sigma_i(\alpha)-\sigma_j(\alpha) =(\sigma_i(\theta)-\sigma_j(\theta))\left(\sum_{k=0}^{d-2}\beta_{k+1}s_{k,i,j}\right).
\end{alignat*}
Hence, we get
\begin{alignat*}1
0\neq \prod_{1\leq i<j\leq d}(\sigma_i(\alpha)-\sigma_j(\alpha))^2 =\gamma^2  \prod_{1\leq i<j\leq d}(\sigma_i(\theta)-\sigma_j(\theta))^2
=\disc_{L/K}(1,\theta,\ldots,\theta^{d-1})\gamma^2.
\end{alignat*}
\end{proof}

\begin{lemma}\label{lem:reduct}
Let  $L_0\subseteq \Qbar$ be a field, let  $(K_i)_{i}$ be a sequence of number fields,
and suppose
\begin{alignat*}1
\alpha\in L_m\backslash L_{m-1} \text{ and } \alpha\in R_n\backslash R_{n-1},
\end{alignat*}
where $L_m$ and $R_n$ are as in (\ref{def: RnLn}). 
Let $\theta_m,\ldots,\theta_n$ be algebraic integers with $K_i=\IQ(\theta_i)$ for $m\leq i\leq n$, and set
$I_i=[\Oseen_{K_i}:\IZ[\theta_i]]$ for $m\leq i\leq n$.
Suppose that $d_i:=[K_i:\IQ]=[L_i:L_{i-1}]$ whenever
 $m< i\leq n$.
Then there exist $\beta_{d_{m}-1,m-1},\ldots,\beta_{1,m-1},\beta_{0,m-1}\in R_{m-1}$ such that 
\begin{alignat*}1
\alpha=\frac{1}{I_mI_{m+1}\cdots I_n}\left(\beta_{d_m-1,m-1}\theta_m^{d_m-1}+\cdots +\beta_{1,m-1}\theta_m+\beta_{0,m-1}\right).
\end{alignat*}
\end{lemma}
\begin{proof}
Since $\Oseen_{K_n}\subset (1/I_n)\IZ[\theta_n]$ it follows that  $R_{n}\subset (1/I_n)R_{n-1}[\theta_n]$. Hence, there exist 
$$\beta_{d_n-1,n-1},\beta_{d_n-2,n-1},\ldots, \beta_{0,n-1}\in R_{n-1}$$
so that
\begin{alignat*}1
\alpha=\frac{1}{I_n}\left(\beta_{d_n-1,n-1}\theta_n^{d_n-1}+\cdots +\beta_{1,n-1}\theta_n+\beta_{0,n-1}\right).
\end{alignat*}
If $n>m$ then $\beta_{d_n-1,n-1}=\cdots=\beta_{1,n-1}=0$ due to the fact that $[L_{n-1}(\theta_n):L_{n-1}]=d_n$, and just as we did for $\alpha$ we can write 
\begin{alignat*}1
\beta_{0,n-1}=\frac{1}{I_{n-1}}\left(\beta_{d_{n-1}-1,n-2}\theta_{n-1}^{d_{n-1}-1}+\cdots +\beta_{1,n-2}\theta_{n-1}+\beta_{0,n-2}\right)
\end{alignat*}
with $\beta_{d_{n-1}-1,n-2},\ldots,\beta_{1,n-2},\beta_{0,n-2}\in R_{n-2}$. Continuing this way we end up with 
\begin{alignat*}1
\alpha=\frac{1}{I_mI_{m+1}\cdots I_n}\left(\beta_{d_m-1,m-1}\theta_m^{d_m-1}+\cdots +\beta_{1,m-1}\theta_m+\beta_{0,m-1}\right)
\end{alignat*}
with $\beta_{d_{m}-1,m-1},\ldots,\beta_{1,m-1},\beta_{0,m-1}\in R_{m-1}$.  
\end{proof}

Next we recall  the following classical result. 
\begin{theorem}[von \.{Z}yli\'{n}ski, \cite{Zylinski}]
\label{thm: Zylinski} Let $K$ be a number field. The field index 
\[
i(K):=\mathrm{gcd}\left\{ [\Oseen_K:\IZ[\theta]]:\theta\in\mathcal{O}_{K},\,\,K=\mathbb{Q}(\theta)\right\} 
\]
has the property that any rational prime $p$ with $p\mid i(K)$ satisfies
$p<[K:\mathbb{Q}]$.
\end{theorem}

Combining Lemma \ref{lem:heightest}, Lemma \ref{lem:reduct}, and Theorem \ref{thm: Zylinski} allows us to prove the following proposition.

\begin{proposition}\label{prop: heightbound}
Let  $L_0\subseteq \Qbar$ be a field, let  $(K_i)_{i}$ be a sequence of number fields,
and suppose
\begin{alignat*}1
L_m=L_{m-1}(\alpha) \text{ and } \alpha\in R_n\backslash R_{n-1},
\end{alignat*}
where $L_m$ and $R_n$ are as in (\ref{def: RnLn}). Let $p$ be a  prime, and suppose that $p\geq d_i=[L_i:L_{i-1}]$ whenever $m\leq i\leq n$.
Write 
$$\Delta=\disc_{L_m/L_{m-1}}(1,\alpha,\ldots,\alpha^{d_m-1}).$$
Then
\begin{alignat*}1
\prod_{v\mid p}\max\{1, |\Delta^{-1}|_v\}^{\frac{d_v}{[L_m:\IQ]}}\geq
p^{\ord_p(\Delta_{K_m})}.
\end{alignat*}
\end{proposition}
\begin{proof}
We apply Lemma \ref{lem:reduct} to  get
\begin{alignat*}1
\alpha=\frac{1}{I_mI_{m+1}\cdots I_n}\left(\beta_{d_m-1,m-1}\theta_m^{d_m-1}+\cdots +\beta_{1,m-1}\theta_m+\beta_{0,m-1}\right).
\end{alignat*}
Since $p\geq d_i$ whenever $m\leq i\leq n$ it
follows from  von Zylinski's Theorem \ref{thm: Zylinski} that we can choose the integral generators 
$\theta_m,\ldots,\theta_n$ such that with $I=I_mI_{m+1}\cdots I_n$
\begin{alignat*}1
p\nmid I.
\end{alignat*}
Next we apply Lemma \ref{lem:heightest} with $K=L_{m-1}$, $L=L_m$, $d=d_m>1$ (note that $L=K(\alpha)$) and
$$
\beta_i=\frac{\beta_{i,m-1}}{I}
$$ 
for all $0\leq i\leq d_m-1$ to get
$$0\neq \Delta=\disc_{L_m/L_{m-1}}(1,\theta_m,\ldots,\theta_m^{d_m-1})\gamma^2,$$
where 
$$\gamma^2=\frac{B^2}{I^{d_m(d_m-1)}},$$ 
and $B$ is a non-zero algebraic integer.

Since $d_{m}=[L_{m}: L_{m-1}]$ and $K_m=\IQ(\theta_m)$ we infer
$$\disc_{L_m/L_{m-1}}(1,\theta_m,\ldots,\theta_m^{d_m-1})=\disc_{K_m/\IQ}(1,\theta_m,\ldots,\theta_m^{d_m-1})=\Delta_{\IZ[\theta_m]},$$
so that we conclude 
\begin{alignat*}1
\Delta=\Delta_{\IZ[\theta_m]}\cdot \frac{B^2}{I^{d_m(d_m-1)}}.
\end{alignat*}
If $p\nmid \Delta_{K_m}$, then the statement follows at once.
If $p\mid \Delta_{K_m}$, then
\begin{alignat*}1
\prod_{v\mid p}\max\{1, |\Delta^{-1}|_v\}^{\frac{d_v}{[L_m:\IQ]}}
\geq \prod_{v\mid p} |\Delta^{-1}|_v^{\frac{d_v}{[L_m:\IQ]}}
=\prod_{v\mid p} \left|\frac{I^{d_m(d_m-1)}}{\Delta_{\IZ[\theta_m]}B^2}\right|_v^{\frac{d_v}{[L_m:\IQ]}}\geq\left|\frac{I^{d_m(d_m-1)}}{\Delta_{\IZ[\theta_m]}}\right|_p\geq p^{\ord_p(\Delta_{K_m})},
\end{alignat*}
as $\Delta_{K_m}\mid \Delta_{\IZ[\theta_m]}$, and $p\nmid I$. 
\end{proof}

We are now ready to prove Theorem \ref{thm: heightdiscbound}.

\begin{proof}[Proof of Theorem \ref{thm: heightdiscbound}]
Let $\alpha\in R$ and $L_{m-1}(\alpha)=L_m$.
Recall that $\Delta=\disc_{L_m/L_{m-1}}(1,\alpha,\ldots,\alpha^{d_m-1})$. 
Using the basic inequalities
$H(a+b)\leq 2H(a)H(b)$, $H(ab)\leq H(a)H(b)$, and the fact that conjugates have equal height, we conclude
\begin{alignat}1\label{ineq: HDeltaalpha}
H(\Delta)=H\left(\prod_{1\leq i<j\leq d_m}(\sigma_i(\alpha)-\sigma_j(\alpha))^2\right)
\leq (\sqrt{2}H(\alpha))^{2d_m(d_m-1)}.
\end{alignat}
Next we note that $H(\Delta)=H(\Delta^{-1})$.
By hypothesis of the theorem and Remark \ref{rem:ldjn} we have $d_i=[K_i:\IQ]=[L_i:L_{i-1}]$ for every $i\in \IN$.
Since $\alpha\notin L_{m-1}$ there exists $n\geq m$ such that
$\alpha\in R_n\backslash R_{n-1}$.
Using Proposition \ref{prop: heightbound} we get
\begin{alignat*}1
H(\Delta)=H(\Delta^{-1})\geq \prod_{p\geq d}\prod_{v\mid p}\max\{1, |\Delta^{-1}|_v\}^{\frac{d_v}{[L_m:\IQ]}}
\geq\prod_{p\geq d}p^{\ord_p(\Delta_{K_m})}
= \frac{|\Delta_{K_m}|}{\prod_{p<d}p^{\ord_p(\Delta_{K_m})}}.
\end{alignat*}
But we have $\ord_p(\Delta_{K_m})\leq C_d$ (see \cite[Theorem B.2.12]{BG}, and using that $d_m\leq d$), and therefore we get
\begin{alignat*}1
H(\Delta)\geq d^{-dC_d}|\Delta_{K_m}|.
\end{alignat*}
Combining the latter with (\ref{ineq: HDeltaalpha}) proves the claim.
\end{proof}

\section{A variation of Theorem \ref{thm: heightdegreebound}}\label{sec: thmghdb}

It would be desirable to replace the hypothesis that the degrees $d_i$ be prime in Theorem \ref{thm: heightdegreebound} by the weaker hypothesis that $L_i/L_{i-1}$ has no proper intermediate field. This can be done if we assume that the family of Galois closures of $L_0, K_1,K_2,K_3,\ldots$ is linearly disjoint over $\IQ$.
\begin{theorem}\label{thm: generalheightdegreebound}
Let $L_0$, $(K_n)_{n}$,$(L_n)_{n}$ and $(d_n)_{n}$ be as in Definition \ref{def notation} (b), and suppose the Galois closures of  $L_0, K_1,K_2,K_3,\ldots$ are linearly disjoint over $\IQ$.
Moreover, suppose there is no proper intermediate field between $L_{n-1}$ and $L_n$ for every $n\in \IN$,
and all degrees  $d_n$ are bounded from above by some $d\in \IN$.
Then there exists $c_d>0$, depending only on $d$, such that for every $\alpha\in R$
\begin{alignat*}1
H(\alpha)\geq
c_d \left(\log[L_0(\alpha):L_0]\right)^{\frac{1}{d^3}}.
\end{alignat*}
\end{theorem}
The proof is almost the same as for Theorem \ref{thm: heightdegreebound} starting by reordering the fields $K_i$
by increasing discriminant $|\Delta_{K_i}|$ and thus redefining the sequence of fields $L_i$. 
Assuming the new family of fields $L_i$ still satisfies the ``no intermediate fields'' condition we can proceed exactly as for the proof of Theorem \ref{thm: heightdegreebound} in the Introduction. The following lemma provides the missing piece. 

\begin{lemma}
Let $L_0$, $(K_n)_{n}$, and $(L_n)_{n}$ be as in Definition \ref{def notation} (b), and suppose the Galois closures of  $L_0, K_1,K_2,K_3,\ldots$ are linearly disjoint over $\IQ$.
Write $K'_n = K_{\sigma(n)}$ for $n\in\IN$, and $L'_n = L_0 \prod_{i=1}^n K'_i$. Then there exists a proper intermediate field between $L_n$ and $L_{n-1}$ for some $n\in\IN$ if and only if there exists a proper intermediate field between $L'_m$ and $L'_{m-1}$ for some $m\in\IN$.
\end{lemma}

\begin{proof}
First we consider the following special case, where $\sigma = \tau_{(1,2)}$ is the transposition $(1,2)$.
\begin{enumerate}
\item (Going up) Suppose $M$ is a proper intermediate field between $L_0$ and $L_1$. 
We claim that $MK_2$ is a proper intermediate field between $L'_1$ and $L'_2$. 
Clearly $L'_1 \subseteq MK_2 \subseteq L'_2$, so it remains to argue that $MK_2$ cannot be $L'_1$ or $L'_2$. If $MK_2 = L'_1 = L_0K_2$, then we have $L_0 \subsetneq M \subseteq L_1 \cap L'_1 = L_1 \cap MK_2$, which contradicts the linear disjointness of $L_0, K_1, K_2$. If $MK_2 = L'_2 = L_2$, then we have $L_1 \subsetneq MK_2$. But then linear disjointness forces $M=L_1$. So the claim holds.
\item (Going down) Suppose $M$ is a proper intermediate field between $L_1$ and $L_2$. 
We write $\widetilde{K_2}$ for the Galois closure of $K_2$, 
and $H$ for the subgroup of $G=\Gal(\widetilde{K_2}L_1/L_1)$ whose fixed field is $L_2$.
Hence there is a subgroup  $H\subsetneq H_M\subsetneq G$ whose fixed field is $M$.
The hypothesis implies that $L_0$, $\widetilde{K_2}$ and $K_1$ are linearly disjoint over $\IQ$, and so
$[\widetilde{K_2}L_1:L_1]=[\widetilde{K_2}L_0:L_0]$. Hence, the injective group homomorphism
$$\phi:\Gal(\widetilde{K_2}L_1/L_1) \to \Gal(\widetilde{K_2}L_0/L_0),$$
sending $\sigma$ to its restriction $\sigma|_{\widetilde{K_2}L_0}$, defines an isomorphism. The fixed fields of 
$\phi(H)$ and $\phi(G)$ are $L'_1$ and $L_0$ respectively. Hence, the fixed field of 
$\phi(H_M)$ is a proper intermediate field between $L_0$ and $L'_1$.
\end{enumerate}

Applying the arguments above with $L_{n-1}, K_n, K_{n+1}$ in place of $L_0, K_1, K_2$ respectively, we find that the existence of a proper intermediate field in the chain $L_0 \subseteq L_1 \subseteq \ldots$ is preserved under transposition $\tau_{(n,n+1)}$. Furthermore, the claims above actually show that for an arbitrary transposition $\tau$, if there is a proper intermediate field between $L_m$ and $L_{m-1}$ for some $m\in\IN$, then there is a proper intermediate field between $L'_{\tau(m)}$ and $L'_{\tau(m)-1}$.

Finally, let $\sigma: \IN\to\IN$ be an arbitrary bijective function, and suppose there exists a proper intermediate field between $L_n$ and $L_{n-1}$ for some $n\in\IN$. Then there exists a bijective function $\rho: \IN\to\IN$ that is the composition of finitely many transpositions, such that $\rho^{-1}(m) = \sigma^{-1}(m)$ for every $m\le \sigma(n)$. Writing $K''_n = K_{\rho(n)}$ for $n\in\IN$ and $L''_n = L_0 \prod_{i=1}^n K''_i$, we have $L'_m = L''_m$ for every $m\le \sigma(n)$, and the claim implies that there is a proper intermediate field between $L''_{\sigma(n)} = L'_{\sigma(n)}$ and $L''_{\sigma(n)-1} = L'_{\sigma(n)-1}$.
\end{proof}

\section{Proof of Proposition \ref{prop:SL1}}\label{sec: qantBZQ}

The proof of Proposition \ref{prop:SL1} uses a result of Nakagawa and a general criterion for the Northcott property \cite[Theorem 3]{WidmerPropN}. For our purposes the following special case is most convenient.  
\begin{corollary}[Widmer, {\cite[Corollary~4]{WidmerPropN}}]\label{cor: cubic ext}
Let $K_0$ be a number field.
Let $K_1,K_2,\ldots$ be a sequence 
of field extensions of $K_0$ such that
$[K_i:K_0]\leq 3$.
Suppose for every $i$
there exists a prime $p_i$ so that
$p_i\mid \Delta_{K_{i}}$ and 
$p_i\nmid \Delta_{K_{j}}$ for $0\leq j<i$.
Then, 
$$
K_0 \prod_{n\geq 1} K_n
$$
has (N).
\end{corollary}

Before proving  Proposition \ref{prop:SL1} let us give a simple construction of a subfield of $\IQ^{(3)}$ with (N) and 
size $\geq 3/5$. To this end let $p>2$ be a prime and consider $D_p(x)=x^3+Ax+B$. We make the ansatz $A=-3n^2$ and $B=2n^3+2p$ and then choose
$n$ to be the even integer closest to $-(p/2)^{1/3}$. Thus, $D_p$ is 
irreducible over $\IQ$ 
and its discriminant 
$$ -4A^3-27 B^2 = -108 p(2 n^3+p) $$ 
is divisible by $p$ and of size $O(p^{5/3})$. In particular,
$p^2$ does not divide the discriminant of $D_p$ once $p>3$ because $108 = 2^2 3^3$.
Moreover, if $q<p$ are both primes,
then $p$ cannot divide the discriminant of $D_q$ since 
otherwise it would be at least $qp>q^2$. 
Setting $K_p=\IQ(\alpha_p)$, where $\alpha_p$
is any root of $D_p$, we conclude that
$p$ ramifies in $K_p$ but not in $K_q$ (for $q$ smaller than $p$) and  $|\Delta_{K_p}|=O(p^{5/3})$. By Corollary \ref{cor: cubic ext},
the composite field of all these cubic fields $K_p$ (one for each prime $p$)
yields a Property $(N)$ field $L$ with $\siz(L)\geq 3/5$. \\

Now let us prove Proposition \ref{prop:SL1}. First we show that for every odd prime $p$, there exists a cubic number field $F_p$ such that $p\mid \Delta_{F_p}$ and $|\Delta_{F_p}|\le 324p$. For $p=3$, let $F_3=\IQ(\alpha)$, where $\alpha$ is a root of $x^3-3x-1$. Then $\Delta_{F_3}=81$. Assume now $p\not=3$. We appeal to the following result.

\begin{theorem}[Nakagawa, {\cite[Theorem~0.4]{Nak}}]
    Let $-n$ be a discriminant of an imaginary quadratic field, $3\nmid n$. Then 
    \[\mathfrak{N}(-n) +\mathfrak{N}(-81n)=3\mathfrak{N}(3n)+1,\]
    where $\mathfrak{N}(d)$ denotes the number of cubic number fields with discriminant $d$.
\end{theorem}

In the theorem above, the right-hand side equals $3\mathfrak{N}(3n)+1\ge 1$, implying that either $\mathfrak{N}(-n)\ge 1$ or $\mathfrak{N}(-81n)\ge 1$, i.e., there exists a cubic number field with discriminant $-n$ or $-81n$. Taking an odd prime $p$, depending on whether $p\equiv 1$ or $3\pmod{4}$, we let $n=4p$ or $p$, respectively. These choices ensure, by Nakagawa's theorem, that there exists a cubic field $F_p$ with $\Delta_{F_p} \in \{-p,-4p,-81p,-324p\}$, so that $p\mid \Delta_{F_p}$ and $|\Delta_{F_p}|\le 324p$, where $324=2^2 3^4$.

Finally, each $\Delta_{F_p}$ is divisible only by $2,3$ and $p$, so $p\nmid \Delta_{F_q}$ for $q<p$ as soon as $p>3$. Hence we apply Corollary \ref{cor: cubic ext} to $K_0=\IQ$ and $K_i=F_{p_i}$. The compositum of all these fields $L$ has Property (N). $L$ contains a cubic field of discriminant at most $X$ for every prime $p\le X/324$, hence $N_L(X)\gg X\log X$ and $\siz(L)=1$, as was required to show.

\bibliographystyle{amsplain}
\bibliography{literature}

\end{document}